\documentclass[12pt]{extarticle}
\usepackage{amsmath, amsthm, amssymb, mathtools, hyperref,color}
\usepackage{graphicx}
\usepackage{comment}
\usepackage{bbm}
\usepackage{algorithm}
\usepackage{algpseudocode}
\usepackage{pifont}
\oddsidemargin \evensidemargin
\newtheorem{theorem}{Theorem}
\numberwithin{theorem}{section}
\newtheorem{proposition}[theorem]{Proposition}
\newtheorem{lemma}[theorem]{Lemma}
\newtheorem{corollary}[theorem]{Corollary}
\newtheorem{definition}[theorem]{Definition}
\newtheorem{remark}[theorem]{Remark}

\newtheorem{problem}[theorem]{Problem}
\newtheorem{example}[theorem]{Example}
\newtheorem{question}{Question}

\date{}
\begin{document}
{
\title{Syzygies over the Polytope Semiring}
\author{Madhusudan Manjunath\footnote{Part of this work was done while the author was at the University of California, Berkeley where he was funded by a Feoder-Lynen Fellowship of the Humboldt Foundation and at the Queen Mary University of London where he was funded by an EPSRC grant. He was also supported by a Leibeniz Fellowship at the Mathematisches Forschungsintitut Oberwolfach and a Junior Research Fellowship at the Erwin Schr\"odinger International Institute for Mathematics and Physics.}}
\maketitle

{\bf Abstract:} Tropical geometry and its applications indicate a ``theory of  syzygies'' over polytope semirings.  Taking cue from this indication, we study a notion of syzygies over the polytope semiring. We begin our exploration with the concept of Newton basis, an analogue of Gr\"obner basis that captures the image of an ideal under the Newton polytope map. The image ${\rm New}(I)$ of a graded ideal $I$ under the Newton polytope map is a graded sub-semimodule of the polytope semiring. Analogous to the Hilbert series,  we define the notion of Newton-Hilbert series that encodes the rank of each graded piece of ${\rm New}(I)$.  We prove the rationality of the Newton-Hilbert series for sub-semimodules that satisfy a property analogous to Cohen-Macaulayness.  We define the notions of regular sequence of polytopes and syzygies of polytopes. We show an analogue of the Koszul property characterizing the syzygies of a regular sequence of polytopes. 

%{\color{blue} TO DO:  i. Blue comments. ii. Referee's comments and response. iii. Spencer's comment on example.}
%%%Introduction.
\section{Introduction}
%We prove semicontinuity results relating the classical and polytope counterparts.  
%%%Concept of Newton Basis.

The concept of Newton polytope  \cite[Chapter 4]{GelKapZel94} of a Laurent polynomial $f$  is a widely studied and useful concept that in many situations, captures important properties of the hypersurface defined by $f$. For an arbitrary subvariety of projective space, this construction is usually generalized to the Chow polytope \cite[Chapter 4]{GelKapZel94} associated to this subvariety.  We undertake a generalization of the Newton polytope in a different direction: given an ideal $I$ of the polynomial ring $\mathbb{K}[x_1,\dots,x_n]$ where $\mathbb{K}$ is a field of infinite cardinality, we associate a space of polytopes  ${\rm New}(I)$ to it. This space is the sub-semimodule of the polytope semiring generated by the Newton  polytope of every element in $I$.

{\bf Newton Basis and Newton-Hilbert Series:}
Let $\mathbb{K}$ be a field of infinite cardinality. Let $\mathcal{A}[n]$ be the polytope semiring whose elements are lattice polytopes with vertices in  $\mathbb{Z}^n_{\geq 0}$, addition $\oplus$ in this semiring is given by convex hull and multiplication is given by Minkowski sum $\odot$ along with the element $0_{\mathcal{A}}$ that is the additive identity and satisfies $P \odot 0_{\mathcal{A}}=0_{\mathcal{A}}$ for all polytopes in $\mathcal{A}$. The only vertex of the ``polytope'' $0_{\mathcal{A}}$ is $0_{\mathcal{A}}$ itself. Depending on the context, we sometimes denote the polytope semiring simply by $\mathcal{A}$.  

Let  ${\rm New}: \mathbb{K}[x_1,\dots,x_n] \rightarrow \mathcal{A}$ be the map taking a non-zero polynomial to its Newton polytope and zero to the element $0_{\mathcal{A}}$. 
We refer to this map as the Newton polytope map. Given an ideal $I$ of $\mathbb{K}[x_1,\dots,x_n]$, the image ${\rm New}(I)$ is a sub-semimodule of $\mathcal{A}$ (see Lemma \ref{Newsub_lem}) that we call the \emph{Newton semimodule} of $I$.  A \emph{Newton basis} of $I$ is a subset $S$ of $I$ such that the set $\{{\rm New}(f)\}_{f \in S}$ generates ${\rm New}(I)$.

\begin{example} \rm
Consider the irrelevant ideal $\mathfrak{m}=\langle x_1,\dots, x_n\rangle$ of $\mathbb{K}[x_1,\dots,x_n]$. The Newton semimodule ${\rm New}(\mathfrak{m})$ of $\mathfrak{m}$ consists of all polytopes in $\mathcal{A}$ except the lattice polytope defined by the origin. The set $\{x_1,\dots,x_n\}$ is a Newton basis for $\mathfrak{m}$. More generally, for a monomial ideal $M$ the monomial minimal generating set $\{m_1,\dots,m_r\}$ forms a Newton basis for $M$.  We shall see in Section \ref{newbas_sect},  Example \ref{infnewbas_ex} that an ideal need not have a finite Newton basis. \qed
\end{example}

The polytope semiring carries a natural $\mathbb{Z}$-grading where the $k$-th graded piece consists of all polytopes in $\mathcal{A} \setminus \{0_{\mathcal{A}}\}$ that are contained in the hyperplane $H_k=\{(x_1,\dots,x_n)|~\sum_i x_i=k\}$ (see Definition \ref{gradsubsem_def} for a precise definition of grading on $\mathcal{A}$).  Suppose that $I$ is a $\mathbb{Z}$-graded ideal then ${\rm New}(I)$ is naturally a $\mathbb{Z}$-graded sub-semimodule of $\mathcal{A}$ and each graded piece is a semigroup $\mathcal{A}_k$ under the operation $\oplus$. This semigroup $\mathcal{A}_k$ has a unique minimal generating set and this set has finite cardinality.

Let $\mathcal{M}$ be a graded sub-semimodule of $\mathcal{A}$. The $k$-th Newton-Hilbert coefficient  $h_{k,\mathcal{M}}^{\rm New}$ of $\mathcal{M}$ is defined as follows.

\begin{center}
$h_{k,\mathcal{M}}^{\rm New}:={\rm rank}(\mathcal{M}_k)$
\end{center}
 where ${\rm rank}(\mathcal{M}_k)$ is the number of minimal generators of $\mathcal{M}_k$, the $k$-th graded piece, as a semigroup (under the operation $\oplus$). This notion is in the same spirit as the Barvinok rank of a matrix \cite{DevSanStu03}.

The Newton-Hilbert Series of $\mathcal{M}$ is the following formal power series.

\begin{center}

$H^{\rm New}_{\mathcal{M}}(t)=\sum_{k=0}^{\infty} h_{k,\mathcal{M}}^{\rm New}t^k$

\end{center}

If $\mathcal{M}={\rm New}(I)$, we denote $H^{\rm New}_{\mathcal{M}}(t)$ simply as $H^{\rm New}_{I}(t)$.  In this context, a natural question is whether the Newton-Hilbert series and the Hilbert series $H_I(t)$ are equal.  For a monomial ideal $M$, we have $H^{\rm New}_{M}(t)=H_M(t)$. But in general we only have ``upper-semicontinuity"  i.e., 

\begin{center}
$h_{k,I}^{\rm New} \geq h_{k,I}$ for all $k$ (Corollary \ref{genlift_cor}).  
\end{center}
where $h_{k,I}^{\rm New}$ and $h_{k,I}$ are the $k$-th Newton-Hilbert and Hilbert coefficient of $I$ respectively.

\begin{example}
\rm
Let $I=\langle x_1-x_2, x_2-x_3,\dots,x_{n-1}-x_n\rangle$ be the toric ideal associated to the root lattice $A_{n-1}$ of type $A$. The Newton semimodule ${\rm New}(I)$ is a graded sub-semimodule and $({\rm New(I)})_1$ is minimally generated as a semigroup by the $\binom{n}{2}$ polytopes $\{C_{i,j}|~i \neq j,1 \leq  i,~j \leq n\}$ where $C_{i,j}$ is the convex hull of $e_i$ and $e_j$ and $e_k$ is the $k$-th standard basis vector of $\mathbb{R}^n$.  In particular, the first Newton-Hilbert coefficient of $I$ is $\binom{n}{2}$. Hence, unlike the case of commutative algebras, the $k$-th Newton-Hilbert coefficient of a sub-semimodule is not necessarily upper bounded by the corresponding Newton-Hilbert coefficient $\binom{n+k-1}{k}$ of the polytope semiring.  \qed
\end{example}

A general problem on Newton-Hilbert series is the following:

\begin{problem}
Classify power series that can occur as the Newton-Hilbert series of sub-semimodules of the polytope semiring.
\end{problem}

We show the rationality of Newton-Hilbert series for sub-semimodules satisfying a property analogous to Cohen-Macaulayness (see Definition \ref{ArtCM_def}).

\begin{theorem}{\bf (Rationality of Newton-Hilbert Series)}\label{ratnewhil_theo}
Let  $\mathcal{M}$ be a Cohen-Macaulay graded sub-semimodule of $\mathcal{A}$, then the Newton-Hilbert series of $\mathcal{M}$ is a rational function.
\end{theorem}

 The Giansiracusa brothers \cite{GianGian13} study the notion of Hilbert polynomial of a tropical variety. For the toric ideal associated to the root lattice $A_n$, their Hilbert polynomial is different from the polynomial underlying the Newton-Hilbert series. In fact, the tropical Hilbert polynomial defined in \cite{GianGian13} coincides with the Hilbert polynomial of the underlying ideal.

{\bf Regular Sequences, Syzygies and Koszul Property for Polytopes:} We formulate homological notions such as regular sequences and syzygies over polytopes. The main challenge in formulating these notions over semirings is the lack of additive inverse. For instance, the concept of ``kernel of a map'' is not well-defined.

For polytopes $P_1,\dots,P_r \in \mathcal{A}$, let $C(P_1,\dots,P_r)$ denote the sub-semimodule in the polytope semiring generated by $P_1,\dots,P_r$.

\indent 
{\bf Regular Sequences:} A sequence $(P_1,\dots,P_r)$ of polytopes is called {\bf regular} if for every $i$ from two to $r$, we have the following property:
  \begin{center} $Q \odot P_i \notin  C(P_1,\dots,P_{i-1})$  for every $Q \notin C(P_1,\dots,P_{i-1})$  \end{center}

\begin{remark}
\rm
Our definition of regular sequences is motivated by the following definition of a regular sequence of elements in an integral domain: elements $r_1,\dots,r_k$ in an integral domain $R$ form a regular sequence if  for every $i$ from two to $k$, the element $r_i$ is a non-zero divisor of $R/\langle r_1,\dots,r_{i-1} \rangle$ i.e., $s \cdot r_i \notin \langle r_1,\dots,r_{i-1} \rangle$ for every $s \notin \langle r_1,\dots,r_{i-1} \rangle$.
\end{remark}

%{\color{blue} discuss the case of two polytopes here?}
%It turns out that for two polytopes, this condition specializes to the following one.

%\begin{center}
 %A pair $(P_1,P_2)$ of polytopes is a regular sequence if and only if they do not share any Minkowski summand other than points.
 %\end{center}
 
% {\color{blue} clarify whether you are working over $\mathcal{A}$ or all lattice polytopes?}
 
We construct examples of regular and non-regular sequences in Section \ref{regseq_sect}. In particular, the sequence of coordinate points is a regular sequence of polytopes. 
Next, we introduce the notion of polytope syzygies. 
 %%In Section \ref{}, we show a Koszul-type property that characterizes a regular sequence of polytopes in terms of its ``syzygies''. 

%%{\color{blue} examples of regular/non-regular sequences of three polytopes and discuss specialization of regular sequences}

\indent

{\bf Polytope Syzygies:}    A {\emph syzygy} of a sequence of polytopes $(P_1,\dots,P_r)$ is an $r$-tuple $(Q_1,\dots,Q_r)$ of elements in $\mathcal{A}$ such that $(P_1\odot Q_1,\dots,P_r \odot Q_r)$ satisfies the following property:  \begin{center} every vertex in $\oplus_{j \in [1,\dots,r]}(P_j\odot Q_j)$ is shared by at least two elements in $\{P_j \odot Q_j\}_{j \in [1,\dots,r]}$. \end{center}  We say that the polytope syzygy is $k$-dimensional if all the polytopes $Q_i$ have dimension exactly $k$.

 \begin{remark}
\rm  This definition is inspired by the notion of tropical linear dependence studied by Jensen and Payne \cite{JenPay14}. A collection $f^{\rm trop}_1,\dots,f^{\rm trop}_r$ of tropical polynomials in $n$-variables are said to be tropically dependent if there exist  $c_1,\dots,c_r \in \mathbb{R}$ such that for every ${\bf x} \in \mathbb{R}^n$ the minimum over $c_1+f^{\rm trop}_1({\bf x}),\dots,c_r+f^{\rm trop}_r({\bf x})$ is attained by at least two elements in $\{c_i+f^{\rm trop}_i({\bf x})\}_{i \in [1,\dots,r]}$. 
This definition is in the same spirit as the notion of tropical rank of a matrix studied in \cite{DevSanStu03}. 

Extending this definition to the notion of syzygy of tropical polynomials, a tuple $(g^{\rm trop}_1,\dots,g_r^{\rm trop})$ of tropical polynomials in $n$-variables is a tropical syzygy if  for every ${\bf x} \in \mathbb{R}^n$ the minimum over $f^{\rm trop}_1({\bf x}) \odot g^{\rm trop}_1({\bf x}),\dots,f^{\rm trop}_r({\bf x})\odot g^{\rm trop}_r({\bf x})$ is attained by at least two elements in $\{f^{\rm trop}_i({\bf x}) \odot g^{\rm trop}_i({\bf x})\}_{i \in [1,\dots,r]}$. In the language of Newton polytopes, this translates to the notion of polytope syzygies. \qed
\end{remark}

Every syzygy $(Q_1,\dots,Q_r)$ has a polytope $W=\oplus_{j \in [1,\dots,r]}(P_j \odot Q_j)$ associated to it. The $r$-tuple $(0_{\mathcal{A}},\dots,0_{\mathcal{A}})$ is always a syzygy since, the corresponding polytope is $0_{\mathcal{A}}$ and this is shared by every element in $\{P_j \odot 0_{\mathcal{A}}\}_{j \in [1,\dots,r]}$.  A natural question in this context is the following: 

\begin{problem}
Fix natural numbers $r,~k \geq 2$, classify polytopes that are associated to a syzygy of polytopes $(P_1,\dots,P_r)$ where $P_i $ are $k$-dimensional and are all distinct. 
\end{problem}

For $r=2$, the answer is precisely those polytopes with vertices in $\mathbb{Z}^n_{\geq 0}$ that are decomposable into a Minkowski sum of two polytopes, both of which have dimension $k$ and similarly, arbitrary polytopes associated to Koszul syzygies are precisely of this form. 
The triangular prism shown in Figure \ref{polysyz} is an example of a polytope associated to a one-dimensional polytope syzygy of three polytopes. We are not aware of a classification of such polytopes. 
For instance, are there numerical invariants that completely characterize this? How does this property depend on the geometry of the polytope? 

The set of all syzygies of $(P_1,\dots,P_r)$ form a semimodule of $\mathcal{A}$ (see Proposition \ref{syzsemimod_prop}). We denote this semimodule by ${\rm Syz}^1(P_1,\dots,P_r)$. The semimodule  ${\rm Syz}^1(P_1,\dots,P_r)$ is not necessarily finitely generated in general, as the following example shows.  

Consider the five polytopes $P_1,\dots,P_5$ where $P_1$ is the convex hull of $(0,0)$ and $(0,1)$ and $P_2$, $P_3$, $P_4$ and $P_5$ are the points $(0,0)$, $(0,1)$, $(1,0)$ and $(1,1)$ respectively. Consider the Minkowski sum of $P_1$ with any line segment joining $(0,0)$ and $(a,b)$  where $(a,b)$ is an integer vector with non-negative coordinates that is not a scalar multiple of $(0,1)$.  This is a parallelogram.  Syzygies of $(P_1,\dots,P_5)$ are obtained by translating the other four polytopes (these are points) to the vertices of this parallelogram.  If $(a,b)$ is a primitive vector, then the corresponding syzygy is a minimal generator of  ${\rm Syz}^1(P_1,\dots,P_5)$ and hence, ${\rm Syz}^1(P_1,\dots,P_5)$ is not finitely generated.

Given an $r$-tuple of polynomials, any syzygy between them  specializes to a polytope syzygy via the Newton polytope map. 

 \begin{proposition}\label{syz_prop}
 Let $(f_1,\dots,f_r)$ be an $r$-tuple of polynomials in $\mathbb{K}[x_1,\dots,x_n]$. Suppose that $(g_1,\dots,g_r)$ is a syzygy of $(f_1,\dots,f_r)$, then $({\rm New}(g_1),\dots,{\rm New}(g_r))$ is a polytope syzygy of $({\rm New}(f_1),\dots,{\rm New}(f_r))$.  \end{proposition}
 
 In fact, a statement stronger than Proposition \ref{syz_prop} holds: every lattice point in $\oplus_{i} ({\rm New}(f_i) \odot {\rm New}(g_i))$ is shared by at least two elements in $\{{\rm New}(f_i) \odot {\rm New}(g_i))\}_i$. The notion of polytope syzygy only captures the ``convex part'' of this  property.

 A general strategy for proving results about syzygies of polynomials is to study (polytope) syzygies of their Newton polytopes and keep track of which polytope syzygies lift.

 The polytope $W$ associated to a polytope syzygy induces a natural equivalence between polytope syzygies. Two syzygies between the same collection of polytopes are said be \emph{equivalent} if their associated polytope is the same and the set of coordinates $k$ where $Q_k=0_{\mathcal{A}}$ is the same for both. In the following, we present some more examples of polytope syzygies: 
 
 \begin{enumerate}
 \item {\bf Koszul Syzygies:} For a pair of polytopes $(P_1,P_2)$, the pair $(P_2,P_1)$ is a syzygy with associated polytope $P_1 \odot P_2$. As in the case of commutative rings, we call this syzygy the \emph{Koszul syzygy} of $P_1$ and $P_2$. Similarly, for a collection $(P_1,\dots,P_r)$ of polytopes and for $1  \leq i<j \leq n$, the  $n$-tuple \begin{center} $K_{i,j}=(0_{\mathcal{A}},\dots,0_{\mathcal{A}},\underbrace{P_j}_{i},\dots,\underbrace{P_i}_{j},0_{\mathcal{A}},\dots,0_{\mathcal{A}})$ \end{center} is a syzygy with associated polytope $P_i \odot P_j$.
 \item \label{example_two} In general, a collection of polytopes has syzygies other than the Koszul syzygies.  Consider the triangular prism shown in Figure \ref{polysyz}. Consider its three quadrilateral faces and let $(P_1,P_2,P_3)$ be the three line segments shown by the pointed curves in Figure  \ref{polysyz}.   The triple $(Q_1,Q_2,Q_3)$, where $Q_2$ and $Q_3$ are the line segments corresponding to the edge of the quadrilateral adjacent to $P_2$ and $P_3$ respectively and $Q_1$ is adjacent to $P_1$ in the remaining quadrilateral, is a syzygy.  This is not a Koszul syzygy and is also not generated by Koszul syzygies.\end{enumerate}
 
%%%Explain the relation to classical syzygies.

{\bf Type of a Syzygy:} Let $(Q_1,\dots,Q_r)$ be a syzygy of polytopes $(P_1,\dots,P_r)$ with associated polytope $W=\oplus_{ j\in [1,\dots,r]}(P_j \odot Q_j)$, then the syzygy is said to be of {\emph type $k$ }, if $k$ is the minimum cardinality of a subset $S$ of $[1,\dots,r]$ such that $W=\oplus_{j \in S}(P_j \odot Q_j)$.

\begin{figure}
  \centering
  \includegraphics[width=8cm]{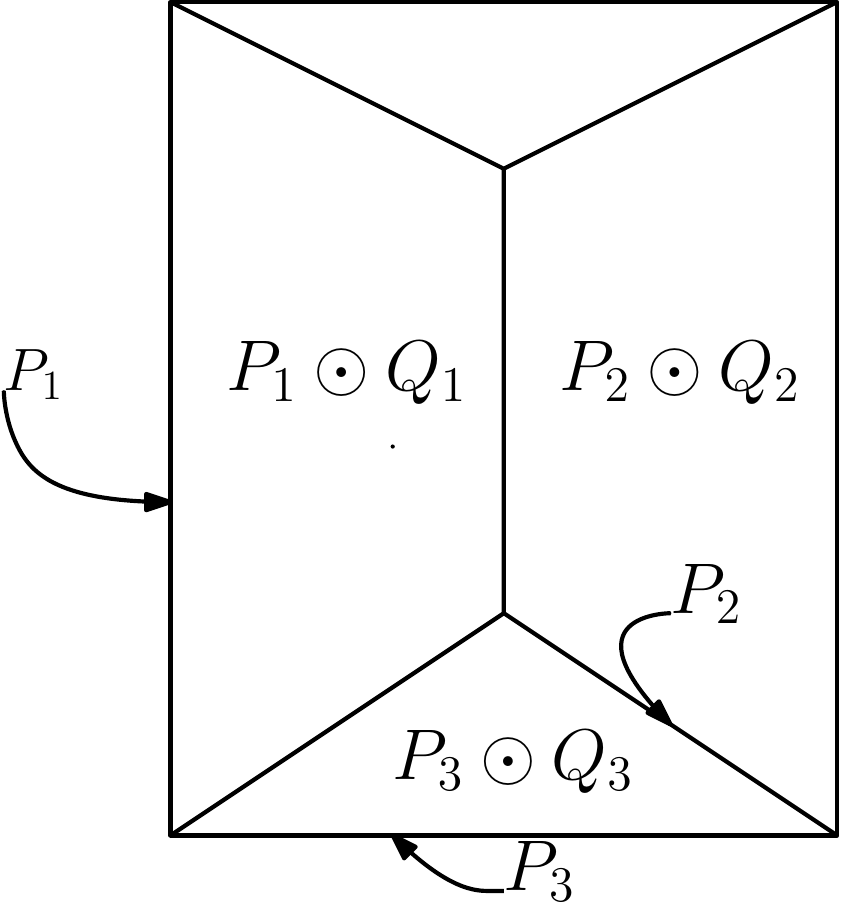}
  \caption{A Triangular Prism: A Polytope Associated with a (Polytope) Syzygy}\label{polysyz}
\end{figure}

Any Koszul syzygy is of type $1$, while the syzygy in Figure \ref{polysyz} is of type $2$. By definition, every syzygy of a collection of $r$ polytopes is of type at most $(r-1)$.

Characterizing syzygies of a collection of polytopes is a guiding question in this context.  In the case of commutative rings, we have the following characterization of syzygies of a regular sequence.

\begin{theorem}{\rm \bf(Koszul Property)}\cite[Proposition 2]{Ser00} For elements $t_1,\dots,t_r$  in a commutative ring $R$, let ${\rm Kos}(t_1,\dots,t_r)$ be the submodule of ${\rm Syz}^1(t_1,\dots,t_r)$ generated by the Koszul syzygies. The equality  ${\rm Syz}^1(t_1,\dots,t_r)={\rm Kos}(t_1,\dots,t_r)$ holds if and only if  $(t_1,\dots,t_r)$ form a regular sequence. \end{theorem}

Let $(P_1,\dots,P_r)$ be a sequence of polytopes.  Let ${\rm Kos}(P_1,\dots,P_r)$ be the sub-semimodule of ${\rm Syz}^1(P_1,\dots,P_r)$ generated by the Koszul syzygies. Suppose that $(Q_1,\dots,Q_r)$ is a syzygy of polytopes $(P_1,\dots,P_r)$, let $W=\oplus_j(P_j \odot Q_j)$ be the polytope associated to it. Suppose that the syzygy is of type-I, then there is an index $i$ such that $W=P_i \odot Q_i$. The set of all  indices for which $W=P_i \odot Q_i$ is called the index set of the type-I syzygy. We show an analogous characterization of regular sequences of polytopes.

\begin{theorem}{\bf (Weak Koszul Property for Polytopes)}\label{kosgenintro_thm}
Let $(P_1,\dots,P_r)$ be a sequence of polytopes in $\mathcal{A}$. Every type-I syzygy $(Q_1,\dots,Q_r)$ of $(P_1,\dots,P_r)$ is equivalent to a type-I syzygy in ${\rm Kos}(P_1,\dots,P_r)$ whose index set contains the index set of $(Q_1,\dots,Q_r)$ if and only if $(P_1,\dots,P_r)$ is a regular sequence.
\end{theorem}

%\begin{theorem}{\bf (Weak Koszul Property for Polytopes)}\label{weakos_thm}
%Let $(P_1,\dots,P_r)$ be a sequence of polytopes in $\mathcal{A}$. Every type-I syzygy of $(P_1,\dots,P_r)$ is equivalent to an element in ${\rm Kos}(P_1,\dots,P_n)$ if and only if $(P_1,\dots,P_r)$ is a regular sequence of polytopes.
%\end{theorem}

On the other hand, the collection $(P_1,P_2,P_3)$ of polytopes in Figure \ref{polysyz} is an example of a regular sequence of polytopes that has a syzygy (for instance, the syzygy in Figure \ref{polysyz}) that is not equivalent to any element in ${\rm Kos}(P_1,P_2,P_3)$.

\subsection{Motivation and Related Work}

The direction of developing linear algebra and algebraic geometry over polytope semirings was suggested by Speyer and Sturmfels in 2009 \cite{SpeStu09}. There is a substantial literature on linear algebra over the tropical semiring, see for example \cite{AkiGauGut08}. 

In this paper, we take first steps towards algebraic geometry over polytope semirings.  We briefly describe previous work that served as impetus for us. One thread was the work of Bayer and Eisenbud on graph curves in 1991 \cite{BayEis91}.  This paper was motivated by Green's conjecture on syzygies of a smooth proper algebraic curve. In particular, Bayer and Eisenbud formulate a conjecture for graph curves analogous to Green's conjecture on a smooth algebraic curve in its canonical embedding. They also proved their conjecture for graph curves where the underlying graph is planar. But, the general case is still open.  However, as they point out in their paper, their conjecture does not imply Green's conjecture. The reason is that the Clifford index of a graph curve of genus $g$ is too small (of the order $\log g$) compared to $\lceil (g-2)/2 \rceil$, the Clifford index of a general proper, smooth curve of genus $g$. This is an obstacle to carrying out standard degeneration arguments.  
 
 Over the past ten years, there has been significant progress in tropical geometry that opens up the possibility of using tropical curves instead of graph curves.  In particular, families of abstract tropical curves with Clifford index (defined in the sense of divisor theory of tropical curves) equal to that of a generic smooth curve have been constructed \cite{CooPayDraRob10}.  Hence, this family of  abstract tropical curves can substitute for graph curves provided that there is a notion of tropical (or polytope) syzygy that behaves ``well'' with respect to degeneration.  A goal of this paper is to serve as a first step in this direction. 

%Jensen and Payne in their recent work on the maximal rank conjecture \cite{} develop the notion of linear independence of tropical polynomials. Our notion of polytope syzygies is inspired by this notion of linear independence. 

Polytope semirings (under the name polytope algebra) appeared in the book of Pachter and Sturmfels \cite[Chapter 2]{PacStu05} in the context of computational biology.  Semiring theory, in particular idempotent semiring theory, has been treated in several books and articles, see for example \cite{Gol13}. Linear algebra over semirings has been the focus of these works. 

Recent works of MacPherson \cite{Macfun15}, \cite{Macpro15} introduce an analogue of integral closure over idempotent semirings and a notion of projective modules over polyhedral semirings.  Our work is another step towards commutative algebra over semirings. Our results are over polytope semirings and do not seem to directly extend to arbitrary semirings. This is primarily because of our use of the Lebseque measure on the set of polytopes.  Other contexts where polytope semirings appear are  Litvinov \cite{Lit10},  Connes \cite[Page 19]{Con15}. Litvinov \cite{Lit10} emphasizes a correspondence principle between classical analysis and idempotent analysis. According to this principle, every result in classical analysis has an idempotent analogue. Theorem \ref{newhilrat_theo} and Theorem \ref{kosgen_thm} confirm a similar interplay between commutative rings and polytope semirings.

Another object related to polytope semirings is McMullen's polytope algebra \cite{Mcm89}. The polytope algebra is the vector space generated by all symbols of the form $[P]$ where $P$ is convex polytope in $\mathbb{R}^n$ along with the relations
\begin{center}$[P \cup Q]+[P \cap Q]=[P]+[Q]$ \end{center}
where $P \cup Q$ is a convex polytope and $[P+{\bf t}]=[P]$ for every ${\bf t} \in \mathbb{R}^n$. 
Multiplication is given by $[P] \cdot [Q]=[P \odot Q]$.

Addition in the polytope algebra seems to be more ``rigid'' than its counterpart in the polytope semiring and hence, the concept of syzygies in the polytope semiring does not seem to have a corresponding object in the polytope algebra. Syzygies in the polytope algebra have been considered in \cite{Mcm93} but we are not aware of a concrete relation between them and the polytope syzygies studied in this paper.

In a recent paper, Rowen \cite{Row16} introduces the concept of negation map on a semiring to handle the problem of the absence of additive inverses in a semiring.  He systematically uses the negation map to define (and sometimes recover) several notions of linear algebra over semirings, the most relevant to this paper is linear dependence. This raises the question of whether we can use Rowen's negation map to define our notion of syzygy in the polytope semiring. At the time of writing this paper, we are not aware of a negation map on the polytope semiring that can recover our notion of polytope syzygy. A simpler question is whether the notion of tropical linear dependence due to Jensen and Payne \cite{JenPay14} is an instance of linear dependence on the tropical semiring with a suitable negation map, we are also not aware of this.

{\bf Acknowledgements:} We thank Matt Baker, Spencer Backman, Alex Fink,  Christian Haase and Bernd Sturmfels for stimulating discussions on the topic of this paper. We thank the anonymous referee for several constructive suggestions.

%{\color{blue} Mention previous work on Polytope semirings and explain the extent to which polytope semirings is essential. Mention Litinov, Sam and Snowden, Pachter and Sturmfels }

\section{Basics of Polytope Semirings}\label{polysemibas_sect}

In this section, we document properties of polytope semirings that we employ throughout the rest of the paper. Let $I$ be an ideal in $\mathbb{K}[x_1,\dots,x_n]$ and consider a polynomial $p \in I$. Let ${\rm New}(p)$ be the image of $p$ under the Newton polytope map and ${\rm New}(I)=\{{\rm New}(p)|~p \in I\}$. By a sub-semimodule $\mathcal{M}$ of $\mathcal{A}$, we mean a subset of $\mathcal{A}$ that is also a semigroup under $\oplus$ and satisfies $\mathcal{A} \odot \mathcal{M} \subseteq \mathcal{M}$.

\begin{lemma}\label{Newsub_lem}
The set ${\rm New}(I)$ is a sub-semimodule of the polytope semiring $\mathcal{A}$.
\end{lemma}
\begin{proof}
Let $P_1={\rm New}(f_1)$ and $P_2={\rm New}(f_2)$ for $f_1,f_2 \in \mathbb{K}[x_1,\dots,x_n]$. We have $P_1 \odot P_2={\rm New}(f_1\cdot f_2)$. Furthermore, for generic choices of $\alpha_1,~\alpha_2 \in \mathbb{K}$, we have $P_1 \oplus P_2={\rm New}(\alpha_1 f_1 +\alpha_2 f_2)$. Since, $\mathbb{K}$ has infinite cardinality such generic choices of $\alpha_1$ and $\alpha_2$ exist. Hence, we conclude that ${\rm New}(I)$ is a semigroup under $\oplus$ and satisfies $\mathcal{A} \odot {\rm New}(I) \subseteq {\rm New}(I)$. Thus, ${\rm New}(I)$ is a sub-semimodule of $\mathcal{A}$.
\end{proof}

\subsection{Graded Sub-semimodules of the Polytope Semiring}
We start by making the notion of a graded sub-semimodule of $\mathcal{A}$ precise. Recall the grading on $\mathcal{A}$ defined in the introduction.

\begin{definition}\label{gradsubsem_def}
A sub-semimodule $\mathcal{M}$ of the polytope semiring is graded if it can be decomposed into a disjoint union of (possibly empty) semigroups $\mathcal{M}_k$ for $k \in \mathbb{Z}$ such that $\mathcal{M}_i \odot \mathcal{A}_j \subseteq \mathcal{M}_{i+j}$ and 
$\{\mathcal{M}_{k}\}_{k \in \mathbb{Z}}$ generates $\mathcal{M}$ as a semigroup under $\oplus$.
\end{definition}

\begin{proposition}
If $I$ is a $\mathbb{Z}$-graded ideal of $\mathbb{K}[x_1,\dots,x_n]$, then ${\rm New}(I)$ is graded as a sub-semimodule of $\mathcal{A}$ with the $k$-th graded piece $({\rm New}(I))_k={\rm New}(I_k \setminus \{0_{\mathbb{K}}\})$ where $0_{\mathbb{K}}$ is the additive identity of $\mathbb{K}$. 
\end{proposition}

%{\color{blue} organize the results of this section better}
\subsection{Measures on Polytope Semirings}

Let $\mu$ be a monotonic measure of $\mathbb{R}^n$ i.e., a measure such that $\mu(S_1) \leq \mu(S_2)$ if $S_1 \subseteq S_2$, for instance the Lebseque measure on $\mathbb{R}^n$. Let $\mu(0_{\mathcal{A}})=0$.
 Proposition \ref{vol-prop} states that a monotonic measure $\mu$ on $\mathbb{R}^n$  is monotonic under linear combinations of elements of $\mathcal{A}$. This property serves as a ``substitute'' for the lack of additive inverse in several arguments, for instance in Proposition \ref{unigen_thm} and the proof of the Koszul property for polytopes.

\begin{proposition} \label{vol-prop}
Suppose that a polytope $P \in \mathcal{A}$ satisfies $P=\oplus_r Q_r$  for a (possibly infinite) family $\{Q_r\}$,  then $Q_r \subseteq P$ for every $Q_r$. Hence,  for any monotonic measure $\mu$ on $\mathbb{R}^d$ we have $\mu(P) \geq \mu(Q_r)$ for every $Q_r$. 
\end{proposition}

 In the following proposition, we use Proposition \ref{vol-prop} to show the uniqueness of minimal generating sets of sub-semimodules in $\mathcal{A}$.

\begin{proposition}\label{unigen_thm}
Every sub-semimodule $\mathcal{M}$ of the polytope semiring has a unique (but not necessarily finite) minimal generating set.  A graded sub-semimodule $\mathcal{M}$ has a unique (but not necessarily finite but countable) graded minimal generating set.
\end{proposition}

\begin{proof}
Suppose that $\{G_r\}_{r \in S_1}$ and $\{H_r\}_{r \in S_2}$ are two distinct minimal generating sets of $\mathcal{M}$.  Suppose that $G_k \notin \{H_r\}_r$. Since $\{H_r\}_r$ is a generating set,  we can write 
\begin{equation}\label{lin-comb1}
G_k=\oplus_r (P_{r} \odot H_r )
\end{equation}
for $P_r \in \mathcal{A}$ over all $r$.  Note that we can assume that $G_k$ is not a translate of any element in $\{H_r\}_{r}$, otherwise this would contradict that $\{G_r\}_{r}$ is a minimal generating set.
Let $\mu$ be the Lebsegue measure on $\mathbb{R}^d$.   By the Brunn-Minkowski inequality, $\mu(P_r \odot H_r) \geq \mu(H_r)$ for all $r \in S$ and by Proposition \ref{vol-prop},  $\mu(G_k) \geq \mu(P_r \odot H_r) $ for all $r \in S$. Furthermore, suppose that $\mu$ is the relative measure of $G_k$, then we have $\mu(G_k)> \mu(H_r)$ for all $r \in S$.   On the other hand, each $H_r$ for $r \in S$ can also be written as an $\mathcal{A}$-linear combination of  $\{G_r\}_r$. Since $\mu(G_k)> \mu(H_r)$ for all $r \in S$,  none of these $\mathcal{A}$-linear combinations involve $G_k$. Combining these linear combinations with (\ref{lin-comb1}), we can write $G_k$ as an $\mathcal{A}$-linear combination of $\{G_r\}_{r} \setminus \{G_k\}$.  This contradicts our assumption that  $\{G_r\}_r$ is a minimal generating set of $\mathcal{M}$.
For a graded sub-semimodule, note that every minimal generating set is graded to conclude the uniqueness of a graded minimal generating set. The countability of a minimal generating set follows by observing that each graded piece has only a finite number of minimal generators.
\end{proof}

\begin{corollary}\label{genlift_cor} Let $I$ be a $\mathbb{Z}$-graded ideal of $\mathbb{K}[x_1,\dots,x_n]$. For every $k \in \mathbb{Z}$, the graded minimal generating set for $({\rm New}(I))_k$ (as a semi-group over $\oplus$) lifts to a generating set for the vector space $I_k$. Hence, for every $k \in \mathbb{Z}$, the $k$-th Newton-Hilbert coefficient of $I$ is at least its $k$-th Hilbert coefficient.
\end{corollary}
\begin{proof}
Let $P_1,\dots,P_r$ be the minimal generating set of $({\rm New}(I))_k$. Consider polynomials $f_j \in I_k$ such that $P_j={\rm New}(f_j)$. We prove the statement by induction on the number of lattice points of ${\rm New}(g)$ where $g \in I_k \setminus \{0_{\mathbb{K}}\}$.  If the number of lattice points of  ${\rm New}(g)$ is one, then $g$ is a monomial and ${\rm New}(g)=P_j$ for some $j$. Hence, $g$ is equal to  $c \cdot f_j$ for some non-zero $c \in \mathbb{K}$. Assume that the statement is true for polynomials in $I_k$ whose Newton polytope contains $\ell \geq 1$ lattice points. Consider a  polynomial $g \in I_k$ such that ${\rm New}(g)$ has $(\ell+1)$ lattice points then, ${\rm New}(g)=\oplus_j {\rm New}(f_j)$. Consider any vertex $v \in {\rm New}(g)$, we know that $v$ is contained in some polytope $P_e$, say. Furthermore, ${\rm New}(f_e) \subseteq  {\rm New}(g) $. Hence, there exists an element $\beta \in \mathbb{K}$ such that ${\rm New}(g-\beta \cdot f_{e})$ contains at most $\ell$ lattice points and hence, by the induction hypothesis $g-\beta \cdot f_e \in I_k$. This implies that $g \in I_k$.  Hence, $\{f_j\}$ is a  generating set of $I_k$ and as a consequence, the $k$-th Newton-Hilbert coefficient of $I$ is at least its $k$-th Hilbert coefficient.
\end{proof}

\begin{corollary}\label{genlift_cor1} Let $I$ be a $\mathbb{Z}$-graded ideal of $\mathbb{K}[x_1,\dots,x_n]$.  The graded minimal generating set for ${\rm New}(I)$ (as a semimodule over $\mathcal{A}$) lifts to a generating set of $I$.  \end{corollary}

\begin{proof}
Let $\{P_r\}_{r \in \mathbb{Z}}$ be the graded minimal generating set of ${\rm New}(I)$.  Consider polynomials $f_j \in I$ such that $P_j={\rm New}(f_j)$. We prove the statement by induction on the number of lattice points of ${\rm New}(g)$  where $g \in I$. If the number of lattice points of  ${\rm New}(g)$ is one then ${\rm New}(g)$  is a translate of $P_e$ for some $e$.  Hence, $g$ is a monomial multiplied by  $f_e$. Assume that the statement is true for polynomials in $I$ whose Newton polytope has $\ell \geq 1$ lattice points. Consider a  polynomial $g \in I$ with $(\ell+1)$ lattice points then, ${\rm New}(g)=\oplus_j ({\rm New}(f_j) \odot Q_j)$ with $Q_j \in \mathcal{A}$ for all $j$. Consider any vertex $v \in {\rm New}(g)$, we know that $v$ is contained in some summand ${\rm New}(f_e) \odot Q_e$, say.  Furthermore, ${\rm New}(f_e) \odot Q_e \subseteq  {\rm New}(g)$. Hence, there exists an element $q_e \in \mathbb{K}[x_1,\dots,x_n]$ such that ${\rm New}(q_e)=Q_e$ and ${\rm New}(g-q_e\cdot f_e)$ contains  at most $\ell$ lattice points (in particular, does not contain $v$) and by the induction hypothesis $g-q_e \cdot f_e \in I$. This implies that $g \in I$.
\end{proof}

\subsection{Minkowski Semigroup}

%{\color{blue} replace ${\rm Min}$ with ${\rm Mink}$?}

The set of lattice polytopes with vertices in $\mathbb{Z}^n_{\geq 0}$ with Minkowski sum as the operation is a semigroup \cite{Sch93}.  We refer to this semigroup as the Minkowski semigroup.  The Minkowski semigroup is commutative i.e., $P \odot Q=Q \odot P$ for every pair $P,~Q$ and cancellative i.e., $Q_1 \odot P_1=Q_1 \odot P_2$ implies that $P_1=P_2$. 

An {\bf irreducible} polytope $P$ in the Minkowski semigroup is a polytope that cannot be written as a Minkowski sum of two polytopes in the Minkowski semigroup neither of which is a point. In particular, every point is a {\bf irreducible} polytope. 

\begin{proposition}\label{uniirr_theo}
Every polytope in the Minkowski semigroup can be written, up to rearrangement, as a  product of irreducible polytopes. This product is not necessarily unique.
\end{proposition}

%The setting of Theorem 3.2.1 is slightly different i.e. over convex bodies  in $\mathbb{R}^n$. But, exactly the same proof also show Theorem \ref{uniirr_theo}.

\subsection{Equations over the Polytope Semiring}\label{eqnssub_subsect}

Fix polytopes $P_1,\dots,P_r$ and $W$ in $\mathcal{A}[n]$. Consider the equation

\begin{equation}\label{semigroup_lem}
W=P_1 \odot Y_1 \oplus P_2 \odot Y_2 \oplus \cdots \oplus P_r \odot Y_r
 \end{equation}
 
 where each $Y_i$ is an element in $\mathcal{A}[n]$.  We ask the following questions:

 \begin{enumerate}
 \item Does this equation have a solution? 
 \item Are there finitely many solutions? 
 \item If yes, can we count them?
 \end{enumerate}
 
 %{\color{blue} More generally, as suggested in \cite{SpeStu09} consider analogue of polynomial equations over $\mathcal{A}$ and seek ``algebraic geometry'' over $\mathcal{A}$.}
 
 Let $\mathcal{M}_W$ be the set of all solutions to this equation. The set $\mathcal{M}_W$ forms a semigroup under coordinate-wise addition.   The following lemma asserts that $\mathcal{M}_W$ is a finite set. 
 
 \begin{lemma}\label{eqfin_lem} The set $\mathcal{M}_W$ is finite. \end{lemma}

  \begin{proof} Let $(Q_1,\dots,Q_r)$ be a solution to Equation (\ref{semigroup_lem}).   We have $P_i \odot Q_i \subseteq W$.  Since $P_i$ and $W$ are lattice polytopes,  there are only a finitely many choices for the polytope $Q_i$. 
  More precisely, if $v$ is a vertex of $P_i$ then the vertices of $Q_i$ are a subset of lattice points of $W$ translated by $-v$.
   \end{proof}

 When $\mathcal{M}_W \neq \emptyset$, Lemma \ref{eqfin_lem} allows us to define \emph{the canonical solution} $\mathcal{C}_W$ to the Equation (\ref{semigroup_lem}) obtained by summing (in other words, taking convex hull) over all elements in $\mathcal{M}_W$.  Furthermore, if we fix another polytope $V$ and consider the set $\mathcal{M}_{W,V}$ of all solutions to Equation (\ref{semigroup_lem}) such that $V$ is a Minkowski summand of each element $Y_i$. This set of solutions is also finite and when $\mathcal{M}_{W,V} \neq \emptyset$,  the \emph{canonical solution} $\mathcal{C}_{W,V}$ with respect to $V$ is  defined as the sum of all elements in $\mathcal{M}_{W,V}$. 
 
 \begin{lemma}\label{cansol_lem}
 Suppose that $V$ is a Minkowski summand of $W$ and let $W=V \odot U$. Assume that $\mathcal{M}_U \neq \emptyset$, the canonical solution $\mathcal{C}_{W,V}$ is equal to $\mathcal{C}_U \odot V$ where $\mathcal{C}_U \odot V$ is the term-wise Minkowski sum of the canonical solution to the equation $U=Y_1 \odot P_1  \oplus \dots  \oplus Y_r \odot P_r$ with $V$.
\end{lemma}
\begin{proof}
The Minkowski sum of any element of $\mathcal{M}_U$ with $V$ is an element in $\mathcal{M}_{W,V}$. Furthermore, if $\mathcal{C}_{W,V}=(R_1,\dots,R_r)$ and suppose that $R_i=L_i \odot V$ for each $i$ from $1$ to $r$. We claim that $(L_1,\dots,L_r)$ is the canonical solution to $U=Y_1 \odot P_1  \oplus \dots  \oplus Y_r \odot P_r$. Otherwise, the canonical solution $\mathcal{C}_U=(L'_1,\dots,L'_r)$ satisfies $L_i \subseteq L'_i$ for every $i$ and hence $R_i=V \odot L_i \subseteq V\odot L'_i$. By the cancellative property of the Minkowski semigroup, $R_j=V \odot L_j \subset V\odot L'_j$ for some $j$ and this contradicts the uniqueness of the canonical solution with respect to $V$. 

\end{proof}

\section{Newton-Hilbert Series}

In this section, we prove Theorem \ref{ratnewhil_theo} of the introduction i.e., the rationality of the Newton-Hilbert series of a graded sub-semimodule of the polytope semiring that satisfies a property analogous to Cohen-Macaulayness.  We start by recalling proofs for rationality of the Hilbert series over the polynomial ring. 

Suppose that $M$ is a finitely generated graded module over the polynomial ring. By the Hilbert syzygy theorem, $M$ has a finite minimal free resolution. The finite minimal free resolution of $M$ is used to express the Hilbert series of $M$ as an alternating sum of the Hilbert series of free modules in each homological degree;  the rationality of the Hilbert series of free modules is a simple computation.  To the best of our knowledge, there is no known analog of the Hilbert syzygy theorem over the polytope semiring. Hence, we do not know if this proof can be adapted to prove the rationality of the Newton-Hilbert series.

Instead, we take cue from a different proof of the rationality of Hilbert series of Cohen-Macaulay modules that goes via induction on the depth of the module.  The base case is that of Artinian modules and is immediate.  Suppose that $r$ is a regular element of a graded module $M$, consider the short exact sequence \begin{center} $0 \rightarrow M \xrightarrow{r} M \rightarrow M/(r \cdot M) \rightarrow 0$ \end{center}

The additivity of Hilbert series in a short exact sequence and the fact that the depth of $M/(r \cdot M)$ is one less than the depth of $M$ then implies the rationality of the Hilbert series of $M$.

In order to adapt this proof for the polytope semiring, we define a notion of regular sequence on a sub-semimodule $\mathcal{M}$. More precisely,  we provide a condition for the sequence of coordinate points (Newton polytopes of variables $x_i$) to be regular.  

Let $(\mathcal{M}_k)_{\parallel e_i}$ be the sub-semigroup of  ${\mathcal{M}}_k$ containing $e_i$ as a Minkowski summand and $(\mathcal{M}_k)_{\perp e_i}$ be the sub-semigroup of $\mathcal{M}_k$ not containing $e_i$ as a Minkowski summand. Note that $\mathcal{M}_k$ is the disjoint union of  $(\mathcal{M}_k)_{\perp e_i}$ and $(\mathcal{M}_k)_{\parallel e_i}$.  In the following, we define the notion of a coordinate point to be a regular element on a sub-semimodule of $\mathcal{A}$.

%Given a subset $S$ of $[1,\dots,n]$, let $\pi_{S}(\mathcal{M})$ be the projection of $\mathcal{M}$ along the coordinates specified by $S$. 

\begin{definition}\label{tropreg_def}
A coordinate point $e_i$ is regular on a graded sub-semimodule $\mathcal{M}$ of $\mathcal{A}$ if and only if for all $k \in \mathbb{N}$ we have: \begin{center} $e_i \odot \mathcal{M}_k=(\mathcal{M}_{k+1})_{\parallel e_i}$. \end{center} 
\end{definition}

The following remarks clarify and motivate Definition \ref{tropreg_def}.

\begin{enumerate}  

\item The containment $e_i \odot \mathcal{M}_k \subseteq (\mathcal{M}_{k+1})_{\parallel e_i}$ always holds and its converse is the non-trivial condition imposed by Definition \ref{tropreg_def}.

\item Definition \ref{tropreg_def} is motivated by its analogue in the commutative ring setting: Let $I$ be a graded ideal of the polynomial ring $\mathbb{K}[x_1,\dots,x_n]$.  A homogenous element $h$ of degree one in the graded polynomial ring is a regular element (i.e., not a zero divisor) on $\mathbb{K}[x_1,\dots,x_n]/I$ if and only if $h \cdot I_k=(I_{k+1})_{\parallel h}$ for every  $k$ where $(I_{k+1})_{\parallel h}$ consists of all elements of degree $k+1$ in $I$ that divide $h$.

\item  Note the difference between the notion of regular sequence in the introduction and Definition \ref{tropreg_def}: in the introduction, we gave a condition for an arbitrary (finite) sequence of polytopes to be a regular sequence on the polytope semiring $\mathcal{A}$ whereas Definition \ref{tropreg_def} is a condition for a coordinate point to be regular on an arbitrary sub-semimodule of the polytope semiring.  According to the following proposition, the two definitions agree when they are both meaningful.
 \end{enumerate}

\begin{proposition} Let $(P_1,\dots,P_r)$ be a regular sequence of polytopes on the polytope semiring $\mathcal{A}$ such that $\langle P_1,\dots,P_r \rangle$ is a graded sub-semimodule of $\mathcal{A}$. The coordinate point $e_i$ is regular (in the sense of Definition \ref{tropreg_def}) on the sub-semimodule generated by $ P_1,\dots,P_r$ if and only if $(e_i,P_1,\dots,P_r)$ is a regular sequence on $\mathcal{A}$.  \end{proposition}

%{\color{blue} add a proof}

\begin{proof}
If  $e_i$ is regular on $\langle P_1,\dots,P_r \rangle$, then suppose that $Q \cdot e_i \in \langle P_1,\dots,P_r \rangle$. Since  $\langle P_1,\dots,P_r \rangle$  is a graded sub-semimodule of $\mathcal{A}$ we know that $Q \cdot e_i=\oplus_j Q'_j$ where each $Q'_j$ is contained in a graded piece of  $\langle P_1,\dots,P_r \rangle$. Furthermore,  each $Q'_j$ contains $e_i$ as a Minkowski summand. Hence, there exists a $Q_j \in \langle P_1,\dots,P_r \rangle$ such that  $Q'_j=e_i \cdot Q_j$. This implies that $Q \in \langle P_1,\dots,P_r \rangle$ showing that $(e_i,P_1,\dots,P_r)$ is a regular sequence.

Conversely, suppose that $(e_i,P_1,\dots,P_r)$ is regular, then if $e_i \cdot Q \in (\langle P_1,\dots,P_r \rangle)_{k+1}$ then $Q \in  (\langle P_1,\dots,P_r \rangle)_{k}$. Hence, $e_i \cdot (\langle P_1,\dots,P_r \rangle)_k=\langle P_1,\dots,P_r \rangle_{k+1}$ for every $k \in \mathbb{N}$.
\end{proof}

%%%{\color{blue} make the regular sequence definition in the intro also a definition?}

We define the notion of regularity for a sequence of coordinate points on an arbitrary sub-semimodule of the polytope semiring. 

\begin{definition}\label{tropregseq_def}
Let $\mathcal{M}$ be a sub-semimodule of the polytope semiring. A sequence $(e_{i_1},\dots,e_{i_r})$ of coordinate points is called a regular sequence on $\mathcal{M}$ if the following conditions are satisfied:
\begin{enumerate}
\item  The coordinate point $e_1$ is regular on $\mathcal{M}$ and $e_{i_{j+1}}$ is regular in $\mathcal{M}^{(j)}:=\mathcal{M} \cap e_{i_1}^{\perp} \cap \dots \cap e_{i_{j}}^{\perp}$ for $j=2,\dots,r-1$ where $e_{i_j}^{\perp}$ is the hyperplane of points with $i_j$-th coordinate zero. 
\item  Let $\mathcal{M}^{(0)}=\mathcal{M}$. For each $k \in \mathbb{N}$ and for $j$ from $0,\dots,r-1$, we have ${\rm rank}((({\mathcal{M}^{(j)}})_{\perp e_{i_{j+1}}})_k)={\rm rank}(({{\mathcal{M}}^{(j+1)}})_k)$.
\end{enumerate}
\end{definition}

\begin{remark}
\rm
 Definition \ref{tropregseq_def} is a condition for a sequence of coordinate points to be regular on an arbitrary graded sub-semimodule of the polytope semiring whereas the notion of regular sequence in the introduction is for an arbitrary sequence of polytopes to be regular on the polytope semiring.  Both these definitions can be applied to a sequence of distinct coordinate points on the polytope semiring. A sequence of distinct coordinate points is a regular sequence according to both Definition \ref{tropregseq_def} and the definition of regular sequences in the introduction. 
\end{remark}

%\begin{remark}
%\rm
%The second condition in Definition \ref{tropregseq_def} is motivated by its counterpart in the case of commutative rings: Suppose $S=\{r_1,\dots,r_k\}$ is a regular sequence of a module $M$ over a commutative ring $R$, the corank of the vector space spanned by elements in $(M/ \langle r_1,\dots,r_{j-1}\rangle \cdot M)_k$ that are contained in the submodule generated by the projection of $r_j$ in $M/ (\langle r_1,\dots,r_{j-1}\rangle \cdot M)$ coincides with the rank of $(M/ (\langle r_1,\dots,r_{j}\rangle \cdot M))_k$.
%\end{remark}

\begin{remark}
\rm
The second condition in Definition \ref{tropregseq_def} is motivated by its counterpart in the case of the polynomial ring: Let $I$ be a graded ideal in the polynomial ring $\mathbb{K}[x_1,\dots,x_n]$.  Suppose that $(x_{i_1},\dots,x_{i_k})$ is a regular sequence of  $\mathbb{K}[x_1,\dots,x_n]/I$. Consider the vector space spanned by elements in $I_d$ obtained by plugging in $x_{i_1}=x_{i_2}=\dots=x_{i_{j-1}}=0_{\mathbb{K}}$. Let $V_j$ be its subspace spanned by elements that are not divisible by $x_{i_j}$.  For every $j$ from two to $k$ and for every $d \in \mathbb{N}$, the rank of $V_j$ is equal to the rank of $(I/ \langle x_{i_1},\dots,x_{i_{j}} \rangle I)_d$. 
\qed
\end{remark}

\begin{definition}\label{ArtCM_def}
A sub-semimodule $\mathcal{M}$ of $\mathcal{A}[n]$ is called \emph {Artinian} if there exists a sufficiently large $k_0 \in \mathbb{N}$ such that $\mathcal{M}_k=\mathcal{A}[n]_k$ for all $k \geq k_0$. Furthermore, it is called \emph{Cohen-Macaulay} if there exists a regular sequence $(e_{i_1},\dots,e_{i_r})$ of coordinate points such that $\mathcal{M}^{(r)}$ is Artinian over the copy $\mathcal{A}[n-r]$ of the polytope semiring generated by the $(n-r)$ coordinates in $[1,\dots,n] \setminus \{i_1,\dots,i_r\}$. The \emph{depth} of $\mathcal{M}$ is the smallest such integer $r$. 
\end{definition}

 In particular, if $\mathcal{M}$ has depth zero then it is Artinian.

%In particular, a sub-semimodule of 

%{\color{blue} explain that an initial degeneration style approach does not seem feasible in this context}

\begin{theorem}{\bf (Rationality of Newton-Hilbert Series)}\label{newhilrat_theo}
Let  $\mathcal{M}$ be a Cohen-Maculay graded sub-semimodule of $\mathcal{A}[n]$, then its Newton-Hilbert series is a rational function. % Furthermore, if the depth of $\mathcal{M}$ is $r$ then the denominator of this rational function is of the form 
\end{theorem}

\begin{proof}
%The proof is analogous to the proof of rationality of the Hilbert series for a Cohen-Macaulay graded ideal (see \cite{}). Let $e_i={\rm New}(x_i)$. 

We apply induction on the depth of $\mathcal{M}$.  If $\mathcal{M}$ is Cohen-Macaulay of depth zero then it is Artinian. Hence, its  Newton-Hilbert series is rational. %Hence, the denominator of the rational function is one. 
  Assume that the  Newton-Hilbert series of every Cohen-Macaulay graded sub-semimodule of depth at most $r$ is rational. Consider a Cohen-Macaulay graded sub-semimodule of depth $(r+1)$. Suppose that $(e_1,\dots,e_{r+1})$ is a regular sequence on $\mathcal{M}$.

Consider ${\mathcal{M}}_k$, the $k$-th graded piece of $\mathcal{M}$. Decompose $\mathcal{M}_k$ into $(\mathcal{M}_k)_{\perp e_1}$ and $(\mathcal{M}_k)_{\parallel e_1}$.

 Hence, $h^{\rm New}_k(\mathcal{M})={\rm rank}((\mathcal{M}_k)_{\perp e_1})+ {\rm rank}((\mathcal{M}_k)_{\parallel e_1})$.  Since $e_1$ is a regular element on $\mathcal{M}$, we have the following two properties:
 \begin{enumerate}
 \item $e_1 \odot \mathcal{M}_{k-1}=(\mathcal{M}_{k})_{\parallel e_1}$.
 \item ${\rm rank}((\mathcal{M}_k)_{\perp e_1})={\rm rank}((\mathcal{M}^{(1)})_k)$.
 \end{enumerate}

We obtain 

\begin{center}  $h_k^{\rm New}(\mathcal{M})=h_{k-1}^{\rm New}(\mathcal{M})+h_{k}^{\rm New}(\mathcal{M}^{(1)})$\end{center}

 This induces the following recurrence on the Newton-Hilbert series: \begin{center} $H^{\rm New}_{\mathcal{M}}(t)=t \cdot H^{\rm New}_{\mathcal{M}}(t)+ H^{\rm New}_{\mathcal{M}^{(1)}}(t)$ \end{center}
 Hence,  $H^{\rm New}_{\mathcal{M}}(t)=H^{\rm New}_{\mathcal{M}^{(1)}}(t)/(1-t)$.
  Note that $\mathcal{M}^{(1)}$ is a sub-semimodule of the copy of the polytope semiring $\mathcal{A}[n-1]$ generated by the last $(n-1)$ coordinates and has depth precisely $r$. By the induction hypothesis,  it  has a rational Newton-Hilbert series.  This implies that $H^{\rm New}_{\mathcal{M}}(t)$ is also rational. %%and the denominator is $(1-t)^{r+1}$.
%In particular, $H^{\rm New}_I(t)=(1-t)^{r+1} \cdot H^{\rm New}_{I/(\langle x_1,\dots,x_{r+1} \rangle) \cdot I}(t)$. 
\end{proof}

Without the assumption of Cohen-Macaulayness, we are not able to show the rationality of the Newton-Hilbert series. The main difficulty is that the induction parameter depth is not apparent in the general case.  In a recent work, Sam and Snowden \cite{SamSno14} develop the concept of combinatorial categories and prove rationality results of the Hilbert series for representation of such categories. We are not aware of a concrete connection between polytope semirings and combinatorial categories. This may be a useful tool to proving rationality results of Newton-Hilbert series.

Remark also that regular elements on $\mathbb{K}[x_1,\dots,x_n]/I$ do not necessarily specialize to regular elements on ${\rm New}(I)$. For instance, consider the toric ideal $I_{A_n}=\langle x_1-x_2,\dots,x_n-x_{n+1} \rangle$ of $\mathbb{K}[x_1,\dots,x_{n+1}]$. The variable $x_1$ is a regular element (since it is not a zero divisor)  on $\mathbb{K}[x_1,\dots,x_{n+1}]/I_{A_n}$.  But $e_1={\rm New}(x_1)$ is not a regular element of ${\rm New}(I_{A_n})$.  For instance, for $k=1$ the second condition in the definition of regular sequence is violated:  ${\rm rank}({\rm New}((I_{A_n})_1)_{\perp e_1})=\binom{n+1}{2}$ whereas  ${\rm rank}((({\rm New}(I_{A_n}))^{(1)})_1)=n$.

%\begin{corollary} There exists a numerical polynomial called the Newton-Hilbert polynomial such that its values agree with ${\rm rank(\mathcal{M}_k)}$ for sufficiently large $k$. \end{corollary}

\subsection{Examples}

\begin{itemize}
\item The  Newton-Hilbert series for the polytope semiring $\mathcal{A}[n]$ agrees with the graded polynomial ring in $n$-variables. It is given by the rational function $1/(1-t)^n$.

\item For any monomial ideal $M$ of $\mathbb{K}[x_1,\dots,x_n]$, the Newton-Hilbert series of ${\rm New}(M)$ coincides with its Hilbert series $M$.

\item  In the following, we present an example of a family of Cohen-Macaulay sub-semimodules of the polytope semiring $\mathcal{A}[3]$, each of depth one. Pick an natural number $d 
\geq 1$. Let $S_{\parallel e_1,d}$ be the set of all points in $H_d=\{ (x_1,x_2,x_3)|~x_1+x_2+x_3=d\}  \cap \mathbb{Z}^3_{\geq 0}$ whose first coordinate is at least one and $S_{\perp e_1,d}$ be its complement in $H_d \cap \mathbb{Z}^3_{\geq 0}$. For each point $q$ in $S_{\perp e_1,d}$, consider the polytope $P_q$ given by the convex hull of $\{q\} \cup S_{\parallel e_1,d}$.  

\indent
Let $\mathcal{M}$ be the sub-semimodule generated by $P_q$ over all $q \in S_{\perp e_1,d}$. By construction, $\mathcal{M}^{(1)}$ is Artinian as a sub-semimodule of the copy of $\mathcal{A}[2]$ corresponding to the semiring generated by the last two coordinates. Furthermore, $\mathcal{M}$ is Cohen-Macaulay of depth one. The same computation as in the proof of Theorem \ref{newhilrat_theo} shows that its Newton-Hilbert series is $t^d/(1-t)^3$. Using a similar approach, we can construct Cohen-Macaulay sub-semimodules of arbitrary depth.

%%\item Consider . {\color{blue} this or another non-trivial example?}
\end{itemize}

%%%Regular Sequences of Polytopes.

\section{Regular Sequence of Polytopes}\label{regseq_sect}

We construct examples of regular (and non-regular) sequences of polytopes. We start with the case of two polytopes.

\begin{proposition} If two polytopes share a non-trivial summand i.e., a Minkowski summand that is not a point then they do not form a regular sequence \end{proposition}

The converse, however is not true. For instance in $\mathbb{R}^2$, let $P_1$ be the rectangle with the four vertices $(0,0),~(1,0),~(0,1),~(1,1)$  and $P_2$ be the triangle with vertices $(0,0)~(1,0),~(1,1)$.  The triangle $P_2$ is irreducible and in particular, the polygons $P_1$ and $P_2$ do not share a non-trivial Minkowski summand. However, the sequence $(P_1,P_2)$ is not regular. To see this, let $P_3$ be the line segment with end points $(0,0)$ and $(1,1)$, and let $P_4$ be the triangle with vertices $(0,0),~(0,1)$ and $(1,1)$. We have $P_1 \odot P_3=P_2 \odot P_4$ and this is the hexagon with vertices $(0,0),~(1,0),~(2,1),~(2,2),~(1,2),~(0,1)$. Furthermore, $P_3$ does not contain $P_2$ as a Minkowski summand
\footnote{The author thanks Klaus Altmann for sharing this example}.

In Section \ref{kosprop_sect}, we give a characterization of regular sequences of polytopes in terms of their syzygies. In the following, we provide an explicit construction of a non-regular and a regular sequence of arbitrary number of polytopes.

%they do not share a non-trivial Minkowski summand i.e., a Minkowski summand that is not a point. \end{proposition}

%{\color{blue} say up to translation. this is false, remove do not share a non-trivial Minkowski summand implies regular part and give example}

%The proof is an immediate consequence of the existence irreducible decomposition (Proposition \ref{uniirr_theo}). 

\subsection{A Non-regular Sequence of Polytopes.} 

Let $P$ be a full-dimensional polytope in $\mathcal{A}[n]$ for $n \geq 2$ and $L$ be a line segment in $\mathcal{A}[n]$. Consider $P \odot L$ and decompose it into $l$-full dimensional simplices $Q_1,\dots,Q_l$, for some $l$.

\begin{proposition} The sequence $(P,Q_1,\dots,Q_l)$ is not a regular sequence. \end{proposition} 
\begin{proof} Note that $L$ has dimension one and is hence, not contained in $C(Q_1,\dots,Q_l)$ where $C(Q_1,\dots,Q_l)$ is the sub-semimodule generated by $Q_1,\dots,Q_l$. By construction,  $P \odot L$ is contained in $C(Q_1,\dots,Q_l)$. \end{proof}

\begin{figure}
  \centering
  \includegraphics[width=8cm]{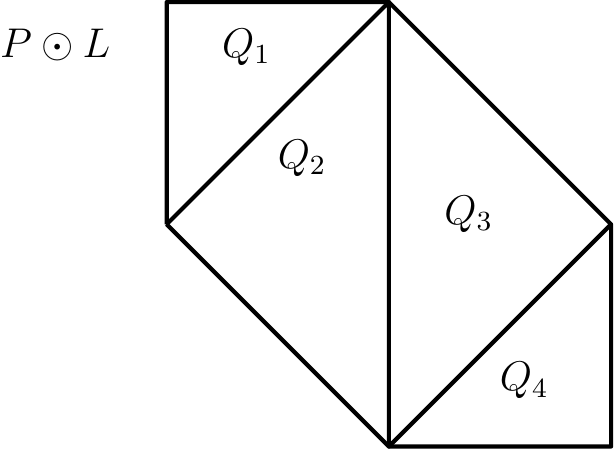}
  \caption{A non-regular sequence of four polytopes}\label{non-regular}
\end{figure}

Figure \ref{non-regular} shows a non-regular sequence of five polytopes.

\subsection{A Regular Sequence of Polytopes}

A basic example of a regular sequence on the polynomial ring is the sequence of variables.  In the following, we note that their Newton polytopes form a regular sequence on the polytope semiring. More precisely, we show the following:

\begin{proposition}
Let $e_1,\dots,e_n$ be the $n$ coordinate points of $\mathbb{R}^n$ i.e., the Newton polytopes of the variables $x_1,\dots,x_n$, the sequence $(e_1,\dots,e_n)$ is a regular sequence of polytopes on the polytope semiring $\mathcal{A}$.
\end{proposition}
 
 \begin{proof}
 Suppose that $Q_i \odot e_i \in C(e_1,\dots,e_{i-1})$ for some $i \geq 2$. We have $Q_i \odot e_i=\oplus_{j=1}^{i-1} (Q_j \odot e_j)$ and $Q_k \odot e_k=Q_k+e_k$ for $k$ from $1$ to $i$. 
 Furthermore, since every point of $Q_i \odot e_i$ has $i$-th coordinate at least one we note that for every $j$ from $1$ to $i-1$, every point of $Q_j$ also has $i$-th co-ordinate at least one. 
 Hence, $Q_j -e_i$ is an element in the polytope semiring $\mathcal{A}$. Finally, we note that $Q_i=\oplus_{j=1}^{i-1}(Q_j-e_i)\odot e_j$. Hence, $Q_i \in C(e_1,\dots,e_{i-1})$ and this concludes the proof of the proposition.
   
 \end{proof}
 
%Since a non-singular linear transformation does not change the property of a regular sequence, we obtain:

%\begin{corollary}
% Let  $b_1,\dots,b_n$  be points in $\mathcal{A}$ whose associated vectors are linearly independent, the sequence $(b_1,\dots,b_n)$ is a regular sequence of polytopes on the polytope semiring $\mathcal{A}$.%
%\end{corollary}

\section{A Koszul Property for Polytopes}\label{kosprop_sect}

In this section, we show the weak Koszul property for polytopes (Theorem \ref{kosgenintro_thm} in the introduction).  The Koszul property for commutative rings can be shown via the exactness of the Koszul complex \cite[Chapter 4]{Ser00}. In the case of polytope semirings, we do not have the notion of Koszul complex and its homology.  On the other hand, we employ the monotonicity of the Lebseque measure under addition in the polytope semiring and its consequences such as the existence of canonical solutions to linear equations discussed in Section \ref{polysemibas_sect}. We believe that these properties can be systematically used as a substitute for the lack additive inverse to establish further analogies between commutative rings and polytope semirings. We first warm up with the case of two polytopes before treating the general case.  We start by proving the following assertion made in the introduction.

\begin{proposition}\label{syzsemimod_prop}The set of all syzygies of $(P_1,\dots,P_r)$ form a semimodule of $\mathcal{A}$.\end{proposition}

\begin{proof} Suppose that $(Q_1,\dots,Q_r)$ and $(R_1,\dots,R_r)$ are syzygies of $(P_1,\dots,P_r)$. Let $W_1=\oplus_i(P_i \odot Q_i)$ and $W_2=\oplus_i(P_i \odot R_i)$. The set of vertices of $W_1 \oplus W_2$ are a subset of the union of the set of vertices of $W_1$ and $W_2$. Hence any vertex $v$ of $W_1 \oplus W_2$ is shared by (at least) two vertices in either $\{P_i \odot Q_i\}_i$ or $\{P_i \odot R_i\}_i$. Suppose that $v$ is shared by $P_1 \odot Q_1$ and $P_2 \odot Q_2$.  Since  $v$ is a vertex of $W_1 \oplus W_2$,  it is a vertex of both $P_1 \odot Q_1 \oplus P_1 \odot R_1$ and  $P_2 \odot Q_2 \oplus P_2 \odot R_2$. Hence, $(Q_1 \oplus R_1,\dots,Q_r \oplus R_r)$ is a syzygy of $(P_1,\dots,P_r)$.

Suppose that $R$ is an element in $\mathcal{A}$.  Suppose that $R=0_{\mathcal{A}}$ then $R \odot (Q_1,\dots,Q_r)=(0_{\mathcal{A}},\dots,0_{\mathcal{A}})$ is a syzygy. Suppose that $R \neq 0_{\mathcal{A}}$,
consider $R \odot W_1$. Any vertex $v$ of $R \odot W_1$ is the sum of a vertex $w$ of $R$ and a vertex $u$ of $W_1$. Since $(Q_1,\dots,Q_r)$ is a syzygy, the vertex $u$ is shared by at least 
two elements in $\{P_i \odot Q_i\}_i$, suppose that they are $P_1\odot Q_1$ and $P_2 \odot Q_2$.  This implies that $v$ is a vertex of both $R \odot P_1 \odot Q_1$ and $R \odot P_2 \odot Q_2$. Hence, $R \odot (Q_1,\dots,Q_r)$ is a syzygy of 
$(P_1,\dots,P_r)$. This concludes the proof of the Proposition.
 \end{proof}

\subsection{Two Polytopes}

\begin{theorem}\label{koztwo_thm}
For two polytopes $P_1,~P_2 \in \mathcal{A}$.  The first syzygy semimodule ${\rm Syz}^{\rm 1}(P_1,P_2)$ is equal to ${\rm Kos}(P_1,P_2)$, the semimodule generated by the Koszul syzygy if and only if $(P_1,P_2)$ is a regular sequence of polytopes on $\mathcal{A}$.  
\end{theorem}
%Otherwise, let $R$ be the maximal Minkowski summand of both $P_1$ and $P_2$ and let $P_1=Q_1 \odot R$ and $P_2=Q_2 \odot R$ then $(Q_1,Q_2)$ is a regular sequence and ${\rm Syz}^{1}(P_1,P_2)$ is generated by $(Q_2,Q_1)$.

\begin{proof}
Note that for every $(P_1,P_2)$, the semimodule generated by the Koszul syzygy $(P_2,P_1)$ is contained in ${\rm Syz}^{1}(P_1,P_2)$.  

Suppose that $(P_1,~P_2)$ is a regular sequence. Consider a syzygy $(W_1,W_2)$ of $(P_1,P_2)$. By definition, we have $P_1 \odot W_1=P_2 \odot W_2$.  Since,  $(P_1, ~P_2)$ is a regular sequence, we know that $W_2=P_1 \odot T_1$ for some $T_1 \in \mathcal{A}$.  Plugging this into the equation $P_1 \odot W_1=P_2 \odot W_2$, and using the cancellative and commutative properties of Minkowski addition, we deduce that $W_1=P_2 \odot T_1$.   Hence, $(W_1,W_2)=T_1 \odot (P_2,P_1)$ and we conclude that $(W_1,W_2)$ is contained in ${\rm Kos}(P_1,P_2)$. 

Suppose that $(P_1,P_2)$ is not regular. By definition, we know that there exists $W_2 \in \mathcal{A}$ that does not contain $P_1$ as a Minkowski summand and satisfies $P_1 \odot W_1=P_2 \odot W_2$ for some $W_1 \in \mathcal{A}$. Hence,  
$(W_1, W_2)$ is a syzygy of $(P_1,P_2)$. We claim that $(W_1,W_2)$ does not belong to ${\rm Kos}(P_1,P_2)$. To see this, note that the contrary implies that $W_2$ contains $P_1$ as a Minkowski summand.

%and $(W_1,W_2)$ is a syzygy of $(P_1,P_2)$. Hence, $W_1 \odot Q_1 \odot R=W_2 \odot Q_2 \odot R$. By the cancellative property of the Minkowski semigroup, we have $W_1 \odot Q_1=W_2 \odot Q_2$.  Since, $R$ is the maximal common Minkowski summand of $P_1$ and $P_2$, we note that $(Q_1$, $Q_2)$ is a regular sequence. Hence,  $(W_1,W_2)$ is contained in the semimodule generated by the Koszul syzygy $(Q_2,Q_1)$.
\end{proof}

\subsection{Arbitrary Number of Polytopes}

 We generalize Theorem \ref{koztwo_thm} to an arbitrary number of polytopes.  Suppose that $(Q_1,\dots,Q_r)$ is a syzygy of polytopes $(P_1,\dots,P_r)$, let $W=\oplus_j(P_j \odot Q_j)$ be the polytope associated to it. Recall that if 
the syzygy is of type-I, then there is an index $i$ such that $W=P_i \odot Q_i$. The set of all  indices for which $W=P_i \odot Q_i$ is called the index set of the type-I syzygy.

%{\color{blue} discuss why the direct generalization does not hold i.e., why equivalence and type-I is necessary and give a counterexample.}

%{\color{blue} discuss the index of a type-I syzygy}

\begin{theorem}{\bf (Weak Koszul Property for Polytopes)}\label{kosgen_thm}
Let $(P_1,\dots,P_r)$ be a sequence of polytopes in $\mathcal{A}$. Every type-I syzygy $(Q_1,\dots,Q_r)$ of $(P_1,\dots,P_r)$ is equivalent to a type-I syzygy in ${\rm Kos}(P_1,\dots,P_r)$ whose index set contains the index set of $(Q_1,\dots,Q_r)$ if and only if $(P_1,\dots,P_r)$ is a regular sequence.
\end{theorem}

We start with an informal explanation of the main ideas used in the proof. 
\indent

{\bf Main Ideas of the Proof:}  We show that if every type-I syzygy is equivalent to an element in ${\rm Kos}(P_1,\dots,P_r)$  then the sequence is regular by contradiction. Suppose that there the sequence $(P_1,\dots,P_r)$ is not regular then there exists an index $i$ and a polytope $Q \notin C(P_1,\dots,P_{i-1})$  such that $P_i  \odot Q \in C(P_1,\dots,P_{i-1})$. This leads to a type-I syzygy $(Q_1,\dots,Q_r)$ that we show is not equivalent to a type-I syzygy in ${\rm Kos}(P_1,\dots,P_r)$ whose index set contains the index set of $(Q_1,\dots,Q_r)$. 

The converse relies crucially on the existence and uniqueness of the canonical solution discussed in Subsection \ref{eqnssub_subsect}. We start with a type-I syzygy $(Q_1,\dots,Q_r)$ of a regular sequence $(P_1,\dots,P_r)$ with associated polytope $W$. For simplicity, assume that  $(Q_1,\dots,Q_r)$ has full support, the general case can be handled by a simple inductive argument.  Since this is a syzygy $W$, we know that for every $i \in [1,\dots,r]$ we have $W=\oplus_{j \neq i}( P_j \odot Q_j)$. Furthermore, we know that $(Q_1,,\dots,Q_r)$ is a type-I syzygy. Hence, we suppose that  $W=P_1 \odot Q_1$. Furthermore, since $(P_1,\dots,P_r)$ is a regular sequence we know that $Q_1 \in C(P_2,\dots,P_r)$. Hence, we know that the set of $S_{W,P_1}$ solutions to $W=P_2 \odot Y_2 \oplus \dots \oplus Y_r \odot P_r$ such that $P_1$ is a Minkowski summand of each term in $(Y_2,\dots,Y_r)$ is not empty.  Using this information, we assume that $(Q_2,\dots,Q_r)$ is a $P_1$-canonical solution to this equation. We express $Q_1$ as an $\mathcal{A}$-linear combination of $(P_2,\dots,P_r)$ via the canonical solution to $Q_1=P_2 \odot Y_2 \oplus \dots \oplus Y_r \odot P_r$  and use the uniqueness of the $P_1$-canonical solution and Lemma \ref{cansol_lem} to construct a type-I syzygy in ${\rm Kos}(P_1,\dots,P_r)$ whose index set contains the index set of $(Q_1,\dots,Q_r)$.
\begin{proof}
%%%
($\Rightarrow$)

Suppose that the sequence $(P_1,\dots,P_r)$ is not regular then there exists an index $i$ and a polytope $Q \notin C(P_1,\dots,P_{i-1})$ such that $P_i  \odot Q \in C(P_1,\dots,P_{i-1})$.  Suppose that $P_i \odot Q= \oplus_{j=1}^{i-1} (P_j \odot R_j)$ then $(R_1,R_2,\dots,R_{i-1},Q,0_{\mathcal{A}},0_{\mathcal{A}},\dots,0_{\mathcal{A}})$ is a syzygy between $\{P_1,\dots,P_i\}$ and the polytope corresponding to it is $P_i \odot Q$. We show that any element equivalent to $(R_1,R_2,\dots,R_{i-1},Q,0_{\mathcal{A}},0_{\mathcal{A}},\dots,0_{\mathcal{A}})$ is not contained in ${\rm Kos}(P_1,\dots,P_r)$.  Assume the contrary and note that the coordinates $j>i+1$ of  $(R_1,R_2,\dots,R_{i-1},Q,0_{\mathcal{A}},0_{\mathcal{A}},\dots,0_{\mathcal{A}})$ are all $0_{\mathcal{A}}$ and this holds for any syzygy equivalent to it. Hence, implies that $(R_1,R_2,\dots,R_{i-1},Q)$ is in ${\rm Kos}(P_1,\dots,P_{i})$.  If there is  a type-I syzygy in ${\rm Kos}(P_1,\dots,P_i)$ that is equivalent to it and whose index set contains the index set of $(R_1,R_2,\dots,R_{i-1},Q)$, then this implies that this syzygy also has $Q$ in its $i$-th coordinate. Furthermore, this implies that $Q \in C(P_1,\dots,P_{i-1})$ and this is a contradiction. Hence, there is no type-I syzygy in ${\rm Kos}(P_1,\dots,P_{r})$ equivalent to $(R_1,R_2,\dots,R_{i-1},Q,0_{\mathcal{A}},0_{\mathcal{A}},\dots,0_{\mathcal{A}})$ and whose index set contains the index set of $(R_1,R_2,\dots,R_{i-1},Q,0_{\mathcal{A}},0_{\mathcal{A}},\dots,0_{\mathcal{A}})$. 

 ($\Leftarrow$)  Suppose that there is a syzygy $(Q_1,\dots,Q_r)$ of a regular sequence $(P_1,\dots,P_r)$ of polytopes.  For simplicity, we assume that the syzygy has full support i.e., no coordinate is $0_{\mathcal{A}}$, this assumption is not necessary, a simple inductive argument will remove this restriction. 
 
 Let $W$ be the polytope corresponding to $(Q_1,\dots,Q_r)$ i.e., $W=P_1\odot Q_1 \oplus P_2 \odot Q_2 \oplus \dots \oplus P_r  \odot Q_r$.  Furthermore, since $(Q_1,\dots,Q_r)$ is a syzygy, every vertex in $W$ is shared by at least two elements in $\{P_i \odot Q_i \}_{i=1}^{r}$. Hence, \begin{center} $W=Q_2 \odot P_2 \oplus Q_3 \odot P_3 \cdots P_r \oplus Q_r$ \end{center}
 
 and $P_1 \odot Q_1 \subseteq W$.
 
Since $(Q_1,\dots,Q_r)$ is a type $1$ syzygy, we have $W=P_i \odot Q_i$ for some $i \in [1,\dots,r]$. We assume that $i=1$ is such an index for  $(Q_1,\dots,Q_r)$.  Since $(P_1 ,\dots,P_r)$ is regular, we conclude that $Q_1 \in C(P_2,\dots,P_r)$. Consider the canonical solution to the equation $Q_1=Y_2 \odot P_2 \oplus \cdots \oplus Y_r \odot P_r$, call it $(L_2,\dots,L_r)$. Hence, $Q_1=L_2 \odot P_2 \oplus L_3 \odot P_2 \cdots \oplus L_r \odot P_r$.  Furthermore, since $\mathcal{M}_{W,P_1} \neq \emptyset$ (recall that $\mathcal{M}_{W,P_1}$ is the set of solutions to $W=Y_2 \odot P_2 \oplus Y_3 \odot P_3 \oplus \dots \oplus Y_r \odot P_r$ with $P_1$ as a Minkowski summand), we assume that $(Q_2,\dots,Q_r)$ is the canonical solution with respect to $P_1$ to the equation $W= P_2  \odot X_2 \oplus P_3 \odot X_3 \cdots \oplus P_r \odot X_r$. This gives a type-I syzygy that is equivalent to $(Q_1,\dots,Q_r)$ and whose same index set contains the index set of $(Q_1,\dots,Q_r)$. Next, we show that it belongs to ${\rm Kos}(P_1,\dots,P_r)$
 
%% Let $W'$ be the minimal generator of $C(P_1 \odot Q_1) \cap C(W)$ {\color{blue} justify its existence}. Hence, $W'=R \odot W=T \odot P_1$ for some polytopes $R$ and $T$ in $\mathcal{A}$. Hence, we obtain
 
 \begin{equation}\label{eqn1}
  Q_1 \odot P_1=  Q_2 \odot P_2 \oplus Q_3 \odot P_3 \dots \oplus  Q_r \odot P_r
 \end{equation}

Plugging into Equation (\ref{eqn1}), we obtain 

\begin{equation}\label{eqn_2}
(L_2\odot P_1) \odot P_2 \oplus (L_3\odot P_1) \odot P_3\oplus \cdots \oplus (L_r \odot P_1) \odot P_r= Q_2 \odot P_2 \oplus  Q_3 \odot P_3 \cdots \oplus  Q_r \odot P_r
 \end{equation}
 
 Using Lemma \ref{cansol_lem}, we conclude that \begin{center} $L_i \odot P_1= Q_i$ for $i$ from $2$ to $r$.\end{center}
 
 Hence, the syzygy $(Q_1,Q_2,\dots,Q_r)=(L_2 \odot P_2 \oplus L_3 \odot P_3 \oplus \cdots \oplus L_r \odot P_r,L_2 \odot P_1,\dots,L_r \odot P_1)$. Hence, this syzygy is in ${\rm Kos}(P_1,\dots,P_r)$. More explicitly, $(Q_1,\dots,Q_r)=L_2 \odot K_{1,2} \oplus L_3 \odot K_{1,3} \oplus \cdots \oplus L_r \odot K_{1,r}$ where $K_{i,j}$ is the Koszul syzygy of the pair $(P_i,P_j)$. 
 
%% We now handle the general case where $W$ is not necessarily equal to $Q_1 \odot P_1$. Suppose that $Q_i \odot P_i$ is contained in the sub-semimodule  spanned by the polytopes $P_j$ where $j \neq i$ we can then decompose $W$ as a convex hull of $P_i \odot Q_i$ over all $i$. Each $P_i \odot Q_i$ corresponds to a Type-I syzygy and apply the previous argument to it {\color{blue} expand this part a bit more, make explicit}. Using this, we conclude that $(Q_1,\dots,Q_r)$ is equivalent to a syzygy contained in  ${\rm Kos}^{\rm trop}(P_1,\dots,P_r)$. 

% {\color{blue}Since $(Q_2,\dots,Q_r)$ is the canonical solution with respect to $P_1$, each $Q_i$ contains $P_1$ as a Minkowski summand. Using the equation  $L_i \odot P_1=R \odot Q_i$, this implies that $L_i$ contains $R$ as a Minkowski summand.  
% Hence, $(Q_1,\dots Q_r)=L'_2 \odot K_{1,2} \oplus L'_3 \odot K_{1,3} \oplus \cdots \oplus L'_r \odot K_{1,r}$ where $L_i=R \odot L'_i$ and hence, $(Q_1,\dots Q_r) \in {\rm Kos}^{\rm trop}(P_1,\dots,P_r)$.}

\end{proof}

Since this characterization is independent of the order of the polytopes, we obtain the following as a corollary:

\begin{corollary}
The property of regular sequence of polytopes does not depend on the order of the polytopes. 
\end{corollary}

A natural generalization of Theorem \ref{kosgen_thm} would be to extend this from type-I syzygies to arbitrary syzygies. The proof of the first implication does not change.    
The example in Figure \ref{polysyz} is a counterexample to the converse. We deduce this by noting that the syzygy does not belong to the semimodule generated by the Koszul syzygies and is the only syzygy in its equivalence class.

{\bf Lifting Property:} A regular sequence of polytopes does not quite satisfy the lifting property. More precisely, if $f_1,\dots,f_r$ are polynomials with generic coefficients and with Newton polytopes $P_1,\dots,P_r$ respectively. If $(P_1,\dots,P_r)$ is a regular sequence then $(f_1,\dots,f_r)$ need not be a regular sequence. For an example, consider the Newton polytopes $P_1,P_2$ and $P_3$ of polynomials $\beta_1 y^3+ \beta_2 x^2y, \gamma_1 y^3+\gamma_2 xy^2, \alpha x \in \mathbb{K}[x,y]$ for coefficients $\alpha, \beta_1,\beta_2,\gamma_1,\gamma_2 \in \mathbb{K}$. The sequence $(P_1,P_2,P_3)$ is a regular sequence of polytopes but the sequence $(f_1,f_2,f_3)$ is not regular for any choice of coefficients. 

 In this case, the reason for the violation of the lifting property is that the Newton sub-semimodule of the ideal generated by $f_1$ and $f_2$ is not equal to the sub-semimodule generated by $P_1$ and $P_2$ but only strictly contains it.  In particular, ${\rm New}(x^2y+xy^2)$ is contained in the former and not the latter.
 
 In order to rectify this problem, we can modify the definition of regular sequence of polytopes as follows. For a sequence of polytopes $P_1,\dots,P_r$, let $D(P_1,\dots,P_r)$  be the Newton sub-semimodule of the ideal generated by a sequence of polynomials $f_1,\dots,f_r$ with ${\rm New}(f_i)=P_i$ and with generic coefficients. Note that $D(P_1,\dots,P_r)$ depends only on the polytopes $P_1,\dots,P_r$.
 A sequence of polytopes $(P_1,\dots,P_r)$ is called strongly regular if $P_i \odot Q \in D(P_1,\dots,P_{i-1})$ implies that $Q \in D(P_1,\dots,P_{i-1})$ for every $i$ from $2$ to $r$. The following question arises: 
 
 \begin{question} Is the Koszul property true for a strongly regular sequence of polytopes i.e., is a  sequence of polytopes strongly regular if and only every syzygy is equivalent to an element in ${\rm Kos}(P_1,\dots,P_r)$?. \end{question}

%%seems more challenging, for instance an arbitrary syzygy does not lead to equations of the form $W=P_1 \odot Y_1 \oplus P_2 \odot Y_2 \oplus \dots \oplus P_n \odot Y_n$. 

\section{Newton Basis}\label{newbas_sect}

Gr\"obner bases are a useful computational tool to study ideals providing a way to degenerate an ideal into a monomial ideal that carries useful information about the original ideal. In this section, we treat an analogue of Gr\"obner basis for the Newton polytope map. We refer to them as Newton basis.

\begin{definition}\rm{({\bf Newton Basis of an Ideal})}
A subset $S$ of polynomials in $I$ is called a {\bf Newton basis} of $I$ if $\{{\rm New}(p)\}_{p \in S}$ generates ${\rm New}(I)$. A Newton basis $S$ is said to be minimal if no strict subset of $S$ is also a Newton basis of $I$. 
\end{definition}

By Proposition \ref{unigen_thm}, we have

\begin{theorem}
Any Newton Basis of $I$ is a generating set of $I$.
\end{theorem}

\begin{algorithm}
\caption{Algorithm for a Minimal Newton Basis}\label{minnewtonbas_alg}

\begin{algorithmic}[1]
\Procedure{Newton Basis} {}

{\bf Input:} A minimal generating set of a graded ideal and an integer $k$.

\\ Let $d_0$ be the minimum degree of any minimal generator.  Set $d=d_0$, $\mathcal{S}_{\rm part}^{(d)}=\emptyset$, $\mathcal{S}^{(d)}=\emptyset$. If $k<d_0$, output empty.

\For{ $d=d_0$, $d=d+1$}

\\  Let $\{n_1,\dots,n_r\}$ be the subset of minimum generators of degree $d$.  Let $\mathcal{S}^{(d)}$ be the union of  $\{n_1,\dots,n_r\}$ with the set $\mathcal{S}_{\rm part}^{(d)}$.

\\ Let $S_i$  be the monomial support of the element $p_i \in \mathcal{S}^{(d)}$. 

\\ For every pair $i \neq j$ and for each monomial $m$ in $S_i \cap S_j$, suppose $\alpha_i$ and $\alpha_j$ are the coefficients of $m$ in $p_i$ and $p_j$ respectively.  Include $\alpha_j \cdot p_i-\alpha_i \cdot p_j$ to 
$\mathcal{S}^{(d)}$.

\\ If $\mathcal{S}^{(d)}$ contains all monomials of degree $d$, then if $d=k$ let $\mathcal{S}^{(k)}$ be the set of all monomials of degree $k$, otherwise let  $\mathcal{S}^{(k)}=\emptyset$ . Exit Iteration. 

\\ Repeat till there are no more such relations between pairs in $\mathcal{S}^{(d)}$.

\\ Compute the minimal subset of  $\mathcal{S}^{(d)}$  according to the partial order given by inclusion on the Newton polytopes. 

\\ Multiply each element of this minimal subset by every variable to obtain the set $\mathcal{S}_{\rm part}^{(d+1)}$. 

\\ Repeat till $d=k$.
%\\ For every pair of monomials $(m_i,m_j) \in S_i \times S_j$ for $i \neq j$, we compute the $S$-pair $({\rm lcm}(m_i,m_j)/m_i,{\rm lcm}(m_i,m_j)/m_i)$. 

%\\ For each linear S-pair, compute $G_{(m_i,m_j)}={\rm lcm}(m_i,m_j)/m_i \cdot  p_j-{\rm lcm}(m_i,m_j)/m_j \cdot p_i$. {\color{blue} include the monomial supports of these elements into $S^{(d+1)}$}

\EndFor

{\bf Output:} The set $\mathcal{S}^{(k)}$, these are the elements of a minimal Newton basis of degree $k$.

\EndProcedure
\end{algorithmic}
\end{algorithm}

The tropical basis of $I$ under the trivial valuation is related to the Newton Basis of $I$ in the following way: a tropical basis $\{f_1,\dots,f_r\}$ of $I$ is a subset of $I$ such that the co-dimension one skeleton of the normal fan of the Newton polytope of any element contains the intersection of the co-dimension one skeleta of the normal fan of the Newton polytopes of the elements in  $\{f_1,\dots,f_r\}$. We do not know a more precise relationship between the two.

%%%II.  Examples of non-finite generation.
\indent Newton basis of monomial ideals has the following simple characterization.

\begin{proposition}
Let $M$ be a monomial ideal, then its (unique) monomial minimal generating set is a minimal Newton basis of $M$. 
\end{proposition}

A natural question in this context is whether every ideal has a finite Newton basis. The answer is no in general as the following example shows.

\begin{example}\label{infnewbas_ex}
Consider the lattice ideal $I_{A_n}=\langle x_1-x_2,x_2-x_3,\dots,x_n-x_{n+1}\rangle$ of $\mathbb{K}[x_1,\dots,x_{n+1}]$ associated to the root lattice $A_n$. The set $\{ {\bf x^{u^+}}- {\bf x^{u^-}}|~{\bf u^{+}}-{\bf u^{-}}\}$ 
where ${\bf u^{+}}-{\bf u^{-}}$ is a primitive vector is the unique minimal Newton basis of $I_{A_n}$ and hence, $I_{A_n}$ has no Newton basis of finite cardinality.
\end{example}

%{
%\color{blue} Point out connection to $\ell_1$ version of theta series of the lattice $A_n$ and how it relates to the Newton-Hilbert series, say that the and counting lattice points in the $\ell_1$-ball.

%}

\begin{proposition}
Any Artinian ideal $I$ has a finite Newton basis.
\end{proposition}
\begin{proof}
Since $I$ is Artinian there exists a sufficiently large degree $d$, say such that all monomials of degree $d$ are contained in $I$. The set of all elements in $I$ of degree at most $d$ is a Newton basis of $I$ and is finite.
\end{proof}

We conclude this section with an algorithm to compute the Newton Basis of a graded ideal.  The main idea is to start with the set of minimal generators of smallest degree and search for linear relations among these generators that ``shrink'' the Newton polytope. For example, if $\{x_1-x_2,x_2-x_3,x_3-x_4\}$ are the set of minimal generators of degree one. We compute the elements $\{x_1-x_3,x_1-x_4,x_2-x_4\}$ whose Newton polytopes does not belong to the sub-semimodule generated by the Newton polytopes of the minimal generators. This gives the elements of the minimal Newton basis of degree one. We then multiply each of these by the variables and take the union with the set of minimal generators of degree two and repeat this procedure to compute elements in the Newton basis of degree two.  Algorithm \ref{minnewtonbas_alg} gives a more precise description.

\footnotesize
\noindent {\bf Author's address:}

\smallskip

\noindent  Department of Mathematics, Indian Institute of Technology Bombay, \\
Powai, Mumbai, Maharashtra 400076, India.\\
{\bf Email id:} m.manjunath@qmul.ac.uk

}\end{document}